\documentclass[12pt, a4paper]{article}
\usepackage[english]{babel}
\usepackage[T1]{fontenc}
\usepackage{amsmath, amsfonts, amssymb, amsthm, mathtools}
\usepackage[numbers]{natbib}
\usepackage{hyperref}
\usepackage{hyperref}
\usepackage{setspace} 
\usepackage{calligra}
\usepackage{lmodern}
\usepackage{lastpage}		
\usepackage{indentfirst}
\usepackage{color}
\usepackage{tcolorbox}
\usepackage{textcomp}
\usepackage{enumerate}
\usepackage{enumitem}
\usepackage[utf8]{inputenc}
\usepackage{array}
\usepackage{pifont}
\usepackage{xspace}
\usepackage{a4wide} 
\usepackage{float}
\usepackage{indentfirst}
\usepackage{url}
\usepackage{csquotes}
\newcommand{\dif}{\mathop{}\!\mathrm{d}}

\newcommand{\wl}{\mathop{}\!\mathrm{w}\hspace{0.00 cm}\text{-}\hspace{-0.17 cm}}

\newtheorem{theorem}{Theorem}[section]  

\newtheorem{lemma}{Lemma}[section]
\newtheorem{definition}{Definition}[section]
\newtheorem{remark}{Remark}[section]     
\newtheorem{example}{Example}[section]

\title{A New Constraint Qualification for Continuous-Time Nonlinear Programming Based on Asymptotic KKT Conditions}

\author{
Rodrigo B. Moreira\thanks{Department of Exact Sciences, State University of Santa Cruz (UESC), Ilhéus, Bahia, Brazil. Email: \texttt{rbmoreira@uesc.br}}
\and
Mois\'es R. C. do Monte\thanks{Pontal Institute of Exact and Natural Sciences, Federal University of Uberlandia (UFU), Ituiutaba, Minas Gerais, Brazil. Email: \texttt{moises.monte@ufu.br}}
\and
Valeriano A. de Oliveira\thanks{Department of Mathematics, Sao Paulo State University (UNESP), São José do Rio Preto, São Paulo, Brazil. Email: \texttt{valeriano.oliveira@unesp.br}}
}
\date{}

\begin{document}
\maketitle

\vspace{-0.4cm}
\begin{abstract}
The asymptotic Karush-Kuhn-Tucker (AKKT) optimality conditions are distinguished from other approaches in the literature by virtue of their capacity to be effectively derived through numerical methods, such as the utilization of an appropriate version of the augmented Lagrange method. These tools are of a theoretical nature, yet they possess practical utility in the identification of candidate solutions to continuous-time programming problems. While this type of optimality condition is valid without imposing any constraint qualification, it is not sufficiently robust to generate good candidate solutions. In some cases, solutions satisfying AKKT conditions are not even stationary. In this study, we investigate conditions that effectively refine this set of candidate solutions with the same precision as the classical Karush-Kuhn-Tucker (KKT) conditions. This is achieved by introducing a novel constraint qualification, designated AKKT-regularity. It has been demonstrated that, under AKKT-regularity, each local optimal solution is shown to satisfy the KKT conditions. In addition, it is demonstrated that this constraint qualification is the weakest possible to ensure such a property. Furthermore, sufficient conditions are provided for its applicability.

\end{abstract}

\maketitle
\section{Introduction}\label{sec1}

The development of continuous-time mathematical programming problems as a formal branch of study can be traced back to the seminal works \cite{bellman:1953,bellman:1957} of Richard Bellman in the 1950s, particularly the introduction of dynamic programming and the Bellman equation. Bellman's methodology provided a robust recursive framework for addressing optimization problems that evolve over continuous time periods. This development represented a substantial advancement from earlier approaches that were rooted in the calculus of variations and control theory.

Contemporary applications of continuous-time mathematical programming and dynamic programming, as pioneered by Bellman, span diverse fields such as engineering process control, finance, economics, and operations research. These models are particularly important for process scheduling, optimal portfolio allocation, resource management, and policy optimization in dynamic, time-dependent environments.

Following the publication of Bellman's works, a significant amount of resources were developed, incorporating a variety of theoretical developments, which ultimately led to the publication of Reiland's papers \cite{reiland1:1980,reiland2:1980,reiland:1980}. Subsequent to this, Zalmai made a significant contribution, as evidenced in the references \cite{zalmai2:1985,zalmai:1985,zalmai:1990}.

The current state of the art in continuous-time optimization encompasses advanced optimality conditions, refined duality theory, and increasingly sophisticated numerical methods. These developments mirror classical results from finite-dimensional optimization while addressing the complexities introduced by the infinite-dimensional and dynamic nature of continuous-time problems. Recent research has established first- and second-order necessary optimality conditions of the Karush-Kuhn-Tucker (KKT) type for continuous-time problems with equality and inequality constraints, thereby generalizing well-known results from nonlinear programming. Such conditions frequently involve regularity and constraint qualifications, including full rank conditions, constant rank conditions and extensions of Mangasarian-Fromovitz constraint qualification. These conditions have become essential tools in analyzing both the theoretical and practical aspects of continuous-time programming. In addressing the concept of duality, significant progress has been made in establishing weak and strong duality theorems for a wide class of continuous-time linear and convex optimization problems. The reader is referred to \cite{deOliveira:2024,monte:2020,monte:2021,monte:2019,Jovic:2021b,Jovic:2022a,Jovic:2022b,Vicanovic:2023b,Vicanovic:2023a,Vicanovic:2025} for theoretical developments and \cite{and:2004b,doMonte2026,pullan:1993,weiss:2008,wen:2012,wu:2016} for numerical methods.

Among the recent developments, one has the sequential optimality conditions, that play a pivotal role in the field of optimization, as they establish the fundamental requirements for optimality, derived from sequences of approximate solutions that converge to the true solution. In contrast to classical pointwise conditions, these conditions do not rely on strict constraint qualifications, thereby ensuring their applicability even in circumstances where such conditions may be violated. These sequential conditions serve as practical stopping criteria for iterative algorithms and ensure convergence to meaningful candidates for optimality. This kind of optimality conditions was introduced in \cite{and:2011} for classic mathematical programming problems in finite dimensions. Subsequently, the aforementioned concept was expanded to encompass optimization problems posed in Banach spaces, as evidenced in the works \cite{Borgens2019, boorgens:2020,Kanzow:2018}. This generalization was further extended to optimal control problems, as discussed in \cite{Moreira2024}, and, most recently, to continuous-time programming, as delineated in \cite{doMonte2026}.

Constraint qualifications play a critical role in sequential optimality conditions by ensuring that points satisfying these sequential conditions also meet classical KKT conditions. While sequential optimality conditions hold at local minimizers regardless of constraint qualifications, the presence of appropriate constraint qualifications (known as strict constraint qualifications) guarantees that limits of sequences satisfying sequential conditions are indeed KKT points. This connection is vital for establishing a link between practical algorithmic stopping criteria and rigorous theoretical optimality guarantees. Furthermore, the identification of the weakest strict constraint qualifications associated with a given sequential optimality condition facilitates the formulation of algorithms that are both broadly applicable and capable of ensuring convergence to meaningful stationary points, despite degeneracies or irregularities in the problem constraints. For instance, in the context of mathematical programming problems, the cone-continuity property (CCP) was introduced in \cite{andreani:2016}. It is the weakest constraint qualification under which solutions satisfying the approximate Karush-Kuhn-Tucker optimality conditions (AKKT) also satisfy the classic KKT conditions. In the context of infinite-dimensional spaces, AKKT-regularity was developed as a generalization of CCP by \cite{boorgens:2020}. In a similar vein, AWMP-regularity was proposed in the context of optimal control problems in \cite{Moreira2026}, as the weakest constraint qualification under which the conditions of the asymptotic weak maximum principle imply the conditions of the classic weak maximum principle. In this paper, we propose a strict type constraint qualification, also called AKKT-regularity, for continuous-time programming. We establish that every solution that satisfies the asymptotic KKT conditions automatically satisfies the KKT conditions when the AKKT-regularity is valid. We further demonstrate that AKKT-regularity constitutes the minimal constraint qualification that ensures this property.

The structure of the paper is delineated as such. Subsequently, a series of fundamental preliminaries are outlined. AKKT-regularity, the new constraint qualification, is proposed in Section \ref{NewCQ}. Section \ref{SuffCrit} is dedicated to give some sufficient criteria to the validity of the new constraint qualification. Conclusions of the paper are given in the last section.

\section{Preliminaries}\label{Preliminaries}

This work considers the following continuous-time programming problem, which incorporates both equality and inequality constraints:
$$
\begin{array}{ll} 
\mbox{minimize} & P(x) = \displaystyle\int_{0}^{T} \phi(x(t),t) \dif t \\ 
\mbox{subject to} & h(x(t),t)= 0 ~ \mbox{a.e.} ~ t \in [0,T], \\
& g(x(t),t) \leq 0 ~ \mbox{a.e.} ~ t \in [0,T], \\
& x \in  L^{\infty}([0,T];\mathbb{R}^{n}),
\end{array}
\eqno{\text{(CTP)}}
$$
where $\phi : \mathbb{R}^{n} \times [0,T] \rightarrow \mathbb{R}$, $h = (h_1,h_2,\ldots,h_p) : \mathbb{R}^{n} \times [0,T] \rightarrow \mathbb{R}^{p}$ and $g= (g_1,g_2,\ldots,g_m) : \mathbb{R}^{n} \times [0,T] \rightarrow \mathbb{R}^m$ are given functions. The set of equality and inequality constraints indices are denoted, respectively, by 
$$
I = \{1,2,\ldots,p\} \quad \text{and} \quad J = \{1,2,\ldots,m\}.
$$

Throughout this paper, all integrals are assumed to be taken in the Lebesgue sense, and inequalities between vectors are to be interpreted component by component. We denote strong, weak, and weak-* convergence by $\to$, $\rightharpoonup$, and $\rightharpoonup^{\ast}$, respectively. 

The feasible set of the problem (CTP) is represented by $\Omega$ and defined as
\begin{align*}
    \Omega = \{ x \in L^{\infty}([0,T];\mathbb{R}^{n}) : h(x(t),t) = 0, ~ g(x(t),t) \leq 0 ~ \text{for a.e.} ~ t \in [0,T]\}.
\end{align*}

Given a real number $r > 0$ and a feasible solution $\bar{x} \in \Omega$, we will denote 
$$
B_{r}(t) = \{ x \in \mathbb{R}^{n} : \| x - \bar{x}(t) \| < r  ~ \mbox{a.e.} ~ t \in [0,T]\}
$$
and 
$$\bar{B}_{r}(t) = \{x \in \mathbb{R}^{n}  :  \| x-\bar{x}(t) \| \leq r \}.$$

The following definition formalizes the concepts of local and global optimal solutions for the continuous-time programming problem (CTP).
\begin{definition} 
\label{local-global} 
    A feasible solution $\bar{x} \in \Omega$ is called a local optimal solution of (CTP) if there exists $r > 0$ such that $P(\bar{x}) \leq P(x)$ for every $x \in \Omega$ satisfying $x(t) \in B_{r}(t)$ for almost every $t \in [0,T]$. Moreover, $\bar{x} \in \Omega$ is referred to as a global optimal solution if $P(\bar{x}) \leq P(x)$ for all $x \in \Omega$.
\end{definition}

\begin{definition}
The Lagrangian function $L : \mathbb{R}^{n} \times \mathbb{R}^{p} \times\mathbb{R}^{m} \times [0,T] \rightarrow \mathbb{R}$ associated to (CTP) is defined as 
$$
L(x,u,v,t) = \phi(x,t) + \sum_{i=1}^{p} u_{i}h_{i}(x,t) + \sum_{j=1}^{m} v_{j}g_{j}(x,t) ~ \mbox{a.e.} ~ t \in [0,T].
$$
\end{definition}

The basic assumptions will be said to be satisfied at $\bar{x} \in \Omega$ if there exists $r>0$ such that:
\begin{itemize}
\item[(H1)] $\phi(\cdot,t)$ is continuously differentiable a.e. $t \in [0,T]$;  
\item[] $\phi(x,\cdot)$ is Lebesgue measurable for each $x$;
\item[] $\vert \phi(\bar{x}(\cdot),\cdot) \vert$ is integrable on $[0,T]$;
\item[] there exist integrable functions $c_{\phi}$ and $\hat{c}_{\phi}$ on $[0,T]$ such that $$\vert \phi(x,t) - \phi(y,t) \vert \leq c_{\phi}(t) \Vert x - y \Vert ~ \forall \, x,y \in \bar{B}_{r}(t) ~ \text{a.e.} ~ t \in [0,T]$$ and $$\Vert \nabla_{x}\phi(x,t)\Vert \leq \hat{c}_{\phi}(t)~ \forall \, x \in \bar{B}_{r}(t) ~ \text{a.e.} ~ t \in [0,T];$$

\item[(H2)] $(h,g)(\cdot,t)$ is continuously differentiable a.e. $t \in [0,T]$; 
\item[] $(h,g)(x,\cdot)$ is Lebesgue measurable for each $x$; 
\item[] $\Vert (h,g)(\bar{x}(\cdot),\cdot) \Vert$ is integrable on $[0,T]$;
\item[] there exist integrable functions $c_{h,g}$ and $\hat{c}_{h,g}$ on $[0,T]$ such that $$\Vert (h,g)(x,t) - (h,g)(y,t) \Vert \leq c_{h,g}(t) \Vert x - y \Vert ~ \forall \, x,y \in \bar{B}_{r}(t) ~ \text{a.e.} ~ t \in [0,T]$$ and $$\Vert \nabla_{x}(h, g)(x,t)\Vert \leq \hat{c}_{h,g}(t)~ \forall \, x \in \bar{B}_{r}(t) ~ \text{a.e.} ~ t \in [0,T].$$
\end{itemize}

\begin{remark}
Clearly, the assertions below are a direct consequence of (H1)--(H2):
\begin{itemize}
\item[(A1)] there exists an integrable function $k_{\phi}$ such that
$$
\vert \phi(x,t) \vert \leq k_{\phi}(t) ~ \forall \, x \in \bar{B}_{r}(t) ~ \mbox{a.e.} ~ t \in [0,T];
$$ 
\item[(A2)] there exists an integrable function $k_{h,g}$ such that
$$
\Vert (h,g)(x,t) \Vert \leq k_{h,g}(t) ~ \forall \, x \in \bar{B}_{r}(t) ~ \mbox{a.e.} ~ t \in [0,T].
$$
\end{itemize}
\end{remark}

Below, we present the traditional (exact) KKT conditions for (CTP). The aforementioned conditions were established under different constraint qualifications, as can be seen in the works of \citeauthor{monte:2019} \cite{monte:2020,monte:2021,monte:2019}.

\begin{theorem} \label{KKT}
Assume that (H1) and (H2) are valid at $\bar{x} \in \Omega$. If $\bar{x}$ is a local optimal solution of (CTP) and a constraint qualification is satisfied, then there exist $u \in L^{\infty}([0,T];\mathbb{R}^p)$ and $v \in L^{\infty}([0,T];\mathbb{R}^m)$ such that, for almost every $t \in [0,T]$,
\begin{eqnarray}
&& \nabla_x L(\bar{x}(t),u(t),v(t),t) = 0, \label{KKT1} \\
&& v_{j}(t)g_{j}(\bar{x}(t),t)=0, ~ j  \in J, \label{KKT2} \\
&& v_j(t) \geq 0, ~ j  \in J. \label{KKT3}
\end{eqnarray}
\end{theorem}

\begin{definition}
Let $\bar{x} \in \Omega$ such that (H1) and (H2) hold. We say that $\bar{x}$ is a KKT solution of (CTP) if there exist $u \in L^{\infty}([0,T];\mathbb{R}^p)$ and $v \in L^{\infty}([0,T];\mathbb{R}^m)$ such that conditions \eqref{KKT1}--\eqref{KKT3} are valid.
\end{definition}

In the sequel, we provide the definitions of AKKT sequences and solutions for (CTP). These definitions were initially established in \citeauthor{doMonte2026} in \cite{doMonte2026}.

\begin{definition} \label{Def_AKKT}
Let $\bar{x} \in \Omega$ such that (H1) and (H2) hold. A sequence $\{(x^{k}, u^{k}, v^{k})\} \subset L^{\infty}([0, T]; \mathbb{R}^{n} \times \mathbb{R}^{p} \times \mathbb{R}^{m})$ is called an \emph{Asymptotic Karush-Kuhn-Tucker} sequence (or simply \emph{AKKT}) for (CTP) at $\bar{x}$ if
\begin{align*}
    &\int_0^T \nabla_{x}L(x^{k}(t),u^{k}(t),v^{k}(t),t) \cdot \gamma(t) \dif t \to 0 ~ \forall \gamma \in L^{\infty}([0,T];\mathbb{R}^{n}), \\ 
    &v_{j}^{k}(t)g_{j}^{-}(x^{k}(t),t) \to 0 ~ \text{a.e.} ~ t \in [0,T], ~ j \in J, \\
    &v_j^{k}(t) \geq 0 ~ \text{a.e.} ~ t \in [0,T], ~ j \in J,
\end{align*}
where $g_j^{-}(x,t) := \max \{ -g_j(x,t),0 \}, ~ j \in J$, a.e. $t \in [0,T]$.
\end{definition}

\begin{remark} \label{remark}
Note that
$$
\int_0^T \nabla_{x}L(x^{k}(t),u^{k}(t),v^{k}(t),t) \cdot \gamma(t) \dif t \to 0 ~ \forall \gamma \in L^{\infty}([0,T];\mathbb{R}^{n})
$$
is equivalent to
$$
\nabla_{x}L(x^{k}(\cdot),u^{k}(\cdot),v^{k}(\cdot),\cdot) \rightharpoonup 0 ~ \text{in} ~ L^1([0,T];\mathbb{R}^n).
$$
\end{remark}

\begin{definition} \label{Def_pw_AKKT}
Let $\bar{x} \in \Omega$ such that (H1) and (H2) are valid. The feasible solution $\bar{x} \in \Omega$ is said to be a \emph{pointwise asymptotic KKT} (or simply \emph{pw-AKKT}) solution for (CTP) if there exists an AKKT sequence $\{(x^{k}, u^{k}, v^{k})\}$ for (CTP) at $\bar{x}$ such that
\begin{align*}
x^{k}(t) \to \bar{x}(t) ~\text{a.e.} ~ t \in [0, T]
\end{align*}
and
\begin{align*}
x^{k}(t) \in B_{r}(t) ~\text{a.e.} ~ t \in [0, T] ~ \forall \, k \in \mathbb{N}.
\end{align*}
\end{definition}

Next, we present the result from \cite{doMonte2026} that asserts the pw-AKKT can be regarded as a genuine necessary optimality condition. Notably, no constraint qualification is imposed.

\begin{theorem} \label{AKKT_opt_cond}
Let $\bar{x}$ be a local optimal solution for (CTP), and suppose that (H1) and (H2) are satisfied. Then, $\bar{x}$ is a pw-AKKT solution for (CTP).
\end{theorem}

\begin{remark}
It is important to note that the Definition \ref{Def_pw_AKKT} differs slightly from the corresponding one given in \citeauthor{doMonte2026} \cite{doMonte2026}. Due to technical reasons, in this study, it was necessary to add the condition $x^{k}(t) \in B_{r}(t)$ for almost every $t \in [0,T]$, for all $k \in \mathbb{N}$. A re-examination of the proof of the preceding theorem (as provided in \cite{doMonte2026}) reveals that, while this condition was not explicitly mentioned in the definition of AKKT solutions in \cite{doMonte2026}, it is nevertheless verified.
\end{remark}

\section{A new constraint qualification}\label{NewCQ}

In this section, the AKKT regularity condition is defined. It is demonstrated that this is a constraint qualification, i.e., under AKKT-regularity, every local optimal solution satisfies the traditional KKT conditions. In addition, we demonstrate that every pw-AKKT solution is indeed a KKT solution, provided that the AKKT-regularity holds. Furthermore, it is demonstrated that AKKT-regularity constitutes the weakest constraint qualification ensuring that every pw-AKKT point is a KKT point.

Despite the fact that the pw-AKKT condition is naturally connected to stopping criteria in practical numerical methods and constitutes a genuine optimality condition, it is still necessary to investigate constraint qualifications ensuring that every pw-AKKT solution is also a KKT solution. The reason is that, for several important classes of optimization problems, the pw-AKKT condition alone may fail to provide a reliable characterization of optimality. In particular, there may exist pw-AKKT solutions that are not related to the actual solution structure of the problem, being neither an optimal solution nor a KKT solution. The following example, inspired by a simplified continuous-time MPCC (\textit{Mathematical Programs with Complementarity Constraints}) setting, illustrates this phenomenon.
\begin{example}
\label{Example_1}
We consider the continuous-time problem below:
\begin{align*}
    \begin{array}{ll}
    \text{minimize} & P(x) = \displaystyle{\int_0^1 x_2(t) \dif t} \\
    \text{subject to} & - x_1(t) \leq 0 ~ \text{a.e.} ~ t \in [0,1], \\
    & x_1(t)x_2(t) = 0 ~ \text{a.e.} ~ t \in [0,1], \\
    & x \in L^\infty([0,1];\mathbb{R}^2).
    \end{array}
\end{align*}
It is not difficult to verify that assumptions (H1) and (H2) are satisfied. Note that $\bar{x}(t) := (0,1)$ for all $t \in [0,1]$ is a pw-AKKT solution. Indeed, this follows by considering the sequence $x_1^{k}(t) = -1/k$, $x_2^{k}(t) = 1$ for all $k \in \mathbb{N}$ and $t \in [0,1]$, and the multipliers $u^{k}(t) = k$ and $v^{k}(t) = k$ for all $k \in \mathbb{N}$ and $t \in [0,1]$. However, the solution $\bar{x}$ does not have a relation to the minimization problem, since it is neither an optimal solution nor a KKT solution. 
\end{example}

Let $\bar{x} \in \Omega$ such that (H1) and (H2) are satisfied. We  define the following sets:
\begin{align*}
    E := \left\{ x \in L^{\infty}([0,T];\mathbb{R}^{n}) : x(t) \in B_{r}(t) \text{ a.e. } t \in [0, T] \right\}
\end{align*}
and
\begin{align*}
    F := \{\theta \in L^{1}([0, T]; \mathbb{R}^{m}) : \theta(t) \geq 0 \text{ a.e. } t \in [0, T]\}.
\end{align*}
We will now introduce the following set-valued map, which plays a central role in the formulation of AKKT-regularity. Furthermore, both pw-AKKT and KKT conditions can be characterized by means of it. The set-valued map in question $\mathcal{M}: E \times F \rightsquigarrow L^{1}([0, T]; \mathbb{R}^{n})$ is given by
\begin{align*}
    \mathcal{M}(x, \theta) &:= \left\{
    \psi \in L^{1}([0, T]; \mathbb{R}^{n}) ~:~ \right.\\
    &\qquad \psi(t) = \sum_{i=1}^{p}u_{i}(t)\nabla_{x}h_{i}(x(t), t) + \sum_{j=1}^{m}v_{j}(t)\nabla_{x}g_{j}(x(t), t) \text{ a.e. } t \in [0, T],\\
    &\qquad (u, v) \in L^{\infty}([0, T]; \mathbb{R}^{p}\times \mathbb{R}^{m}), \\
    &\qquad \left.v_{j}(t)g_{j}^{-}(x(t), t) = \theta_{j}(t), ~ v_{j}(t) \geq 0, ~ j \in J, \text{ a.e. } t \in [0, T] \right\}.
\end{align*}

Next lemma presents an alternative formulation of the KKT conditions using the set-valued map $\mathcal{M}$.
\begin{lemma} \label{Lem_M_KKT}
Assume that (H1) and (H2) are satisfied at $\bar{x} \in \Omega$. Then, $\bar{x}$ is a KKT solution of (CTP) if and only if
\begin{align*}
-\nabla_{x}\phi(\bar{x}(\cdot), \cdot) \in \mathcal{M}(\bar{x}, 0).
\end{align*}
\end{lemma}
\begin{proof}
Assume that $\bar{x}$ is a KKT solution of (CTP). By Theorem \ref{KKT}, there exists $(u,v) \in L^{\infty}([0,T];\mathbb{R}^{p} \times \mathbb{R}^{m})$ such that, for almost every $t \in [0,T]$,
\begin{align}
& \nabla_{x}\phi(\bar{x}(t), t) + \sum_{i=1}^{p}u_{i}(t)\nabla_{x}h_{i}(\bar{x}(t), t) + \sum_{j=1}^{m}v_{j}(t)\nabla_{x}g_{j}(\bar{x}(t), t) = 0, \label{Eq_1} \\
& v_{j}(t)g_{j}(\bar{x}(t), t) = 0, ~ j \in J, \label{Eq_2} \\
& v_{j}(t) \geq 0, ~ j \in J. \label{Eq_3}
\end{align}
By using the feasibility of $\bar{x}$ and \eqref{Eq_2}, we have, for almost every $t \in [0, T]$,
\begin{align*}
v_{j}(t) g_{j}^{-}(\bar{x}(t), t) = v_{j}(t) \max\{-g_{j}(\bar{x}(t),t), 0\} = - v_{j}(t) g_{j}(\bar{x}(t), t) = 0, ~ j \in J.
\end{align*}
From \eqref{Eq_1} and \eqref{Eq_3}, it follows that
\begin{align*}
-\nabla_{x}\phi(\bar{x}(\cdot), \cdot) \in \mathcal{M}(\bar{x}, 0).
\end{align*}

Conversely, assume that $-\nabla_{x}\phi(\bar{x}(\cdot), \cdot) \in \mathcal{M}(\bar{x},0)$. From the definition of $\mathcal{M}(\bar{x},0)$, we know that there exists $(u, v) \in L^{\infty}([0, T]; \mathbb{R}^{p}\times \mathbb{R}^{m})$ such that, for almost every $t \in [0, T]$,
\begin{align}
& -\nabla_{x}\phi(\bar{x}(t), t) = \sum_{i=1}^{p}u_{i}(t)\nabla_{x}h_{i}(\bar{x}(t), t) + \sum_{j=1}^{m}v_{j}(t)\nabla_{x}g_{j}(\bar{x}(t), t), \label{Eq_4} \\
& v_{j}(t)g_{j}^{-}(\bar{x}(t), t) = 0, ~ j \in J, \label{Eq_5} \\
& v_{j}(t) \geq 0, ~ j \in J. \label{Eq_6}
\end{align}
Using the feasibility of $\bar{x}$ and \eqref{Eq_5}, we obtain, for almost every $t \in [0, T]$,
\begin{align} \label{eqtn}
-v_{j}(t) g_{j}(\bar{x}(t), t) = v_{j}(t) \max\{-g_{j}(\bar{x}(t), t), 0\} = v_{j}(t) g_{j}^{-}(\bar{x}(t), t) = 0, ~ j \in J.
\end{align}
It follows from \eqref{Eq_4}, \eqref{Eq_6} and \eqref{eqtn} that $\bar{x}$ is a KKT solution of (CTP).
\end{proof}

The Lemma \ref{Lema_pw_AKKT}, stated and proved below, ensures that the pw-AKKT conditions can be represented in an alternative form involving the weak Painlevé–Kuratowski outer (or upper) limit of the set-valued map $\mathcal{M}: E \times F \rightsquigarrow L^{1}([0, T]; \mathbb{R}^{n})$ given by
\begin{align*}
    \wl\limsup_{(x, \theta) \to (\bar{x}, 0)}\mathcal{M}(x, \theta) &:= \left\{ \psi \in L^{1}([0, T]; \mathbb{R}^{n}) ~:~  \exists ~ \{(x^{k}, \theta^{k})\} \subset E \times F ~ \text{with} \right.\\
    &\qquad (x^{k}(t), \theta^{k}(t)) \to (\bar{x}(t), 0) \text{ a.e. } t \in [0, T], \\
    &\qquad \left.\psi^{k} \rightharpoonup \psi \text{ in } L^{1}([0, T]; \mathbb{R}^{n}), ~ \psi^{k} \in \mathcal{M}(x^{k}, \theta^{k}) ~ \forall \, k \in \mathbb{N} \right\}.
\end{align*}

\begin{lemma} \label{Lema_pw_AKKT}
Assume that (H1) and (H2) are satisfied at $\bar{x} \in \Omega$. Then, $\bar{x}$ is a pw-AKKT solution of (CTP) if and only if
\begin{align*}
    -\nabla_{x}\phi(\bar{x}(\cdot), \cdot) \in \wl\limsup_{(x, \theta) \to (\bar{x}, 0)}\mathcal{M}(x, \theta).
\end{align*}
\end{lemma}
\begin{proof}
Assume that $\bar{x}$ is a pw-AKKT solution of (CTP). By Definitions \ref{Def_AKKT} and \ref{Def_pw_AKKT} and Remark \ref{remark}, there exist sequences $\{(x^{k}, u^{k}, v^{k})\} \subset L^{\infty}([0, T]; \mathbb{R}^{n}\times \mathbb{R}^{p}\times \mathbb{R}^{m})$ and $\{(\varepsilon^{k}, \theta^{k})\} \subset L^{1}([0, T]; \mathbb{R}^{n})\times L^{1}([0, T]; \mathbb{R}^{m})$ such that, for each $k \in \mathbb{N}$ and almost every $t \in [0, T]$,
\begin{align}
    & \nabla_{x}\phi(x^{k}(t), t) + \sum_{i=1}^{p}u_{i}^{k}(t)\nabla_{x}h_{i}(x^{k}(t), t) + \sum_{j=1}^{m}v_{j}^{k}(t)\nabla_{x}g_{j}(x^{k}(t), t) = \varepsilon^{k}(t), 
    \label{Eq_7} \\
    & v_{j}^{k}(t)g_{j}^{-}(x^{k}(t), t) = \theta_{j}^{k}(t), ~ j \in J, 
    \label{Eq_8} \\
    & v_{j}^{k}(t) \geq 0, ~ j \in J, 
    \label{Eq_9}\\
    &x^{k}(t) \in B_{r}(t), 
    \label{Eq_Extra_x_1}
    \end{align}
where $\varepsilon^{k} \rightharpoonup 0$ in $L^{1}([0, T]; \mathbb{R}^{n})$ and $(x^{k}(t), \theta^{k}(t)) \to (\bar{x}(t), 0)$ for almost every $t \in [0, T]$. By \eqref{Eq_8}, \eqref{Eq_9} and \eqref{Eq_Extra_x_1}, we conclude that $(x^{k}, \theta^{k}) \in E \times F$ for each $k \in \mathbb{N}$. Clearly, $\left( \varepsilon^{k} - \nabla_x \phi(x^{k}(\cdot),\cdot) \right) \in \mathcal{M}(x^{k},\theta^{k})$ for all $k \in \mathbb{N}$. Moreover, for all $\gamma \in L^{\infty}([0, T]; \mathbb{R}^{n})$, we have
\begin{align*}
\int_{0}^{T}\left[\varepsilon^{k}(t) - \nabla_{x}\phi(x^{k}(t), t)\right]\cdot \gamma(t)\dif t = \int_{0}^{T}\varepsilon^{k}(t) \cdot \gamma(t)\dif t + \int_{0}^{T}\left[- \nabla_{x}\phi(x^{k}(t), t)\cdot \gamma(t) \right] \dif t
\end{align*}
for each $k \in \mathbb{N}$. Taking the limit as $k \to \infty$, applying the Dominated Convergence Theorem \cite[Theorem 3.25]{Gordon1994} and the continuity of $\phi(\cdot, t)$, it follows that
\begin{align*}
\int_{0}^{T}\left[\varepsilon^{k}(t) - \nabla_{x}\phi(x^{k}(t), t)\right]\cdot \gamma(t)\dif t \to \int_{0}^{T}\left[- \nabla_{x}\phi(\bar{x}(t), t)\cdot \gamma(t) \right] \dif t
\end{align*}
for all $\gamma \in L^{\infty}([0, T]; \mathbb{R}^{n})$, i.e., $\left(\varepsilon^{k} - \nabla_{x}\phi(x^{k}(\cdot), \cdot)\right) \rightharpoonup - \nabla_{x}\phi(\bar{x}(\cdot), \cdot)$ in $L^{1}([0, T]; \mathbb{R}^{n})$. Therefore,
\begin{align*}
-\nabla_{x}\phi(\bar{x}(\cdot), \cdot) \in \wl\limsup_{(x, \theta) \to (\bar{x}, 0)}\mathcal{M}(x, \theta).
\end{align*}

On the other hand, assume that the inclusion above is true. By definition, there exist sequences $\{(x^{k}, \theta^{k})\} \subset E \times F$ and $\{\psi^{k}\} \subset L^{1}([0, T]; \mathbb{R}^{n})$ such that $x^{k}(t) \in B_{r}(t)$ for all $k \in \mathbb{N}$ and $(x^{k}(t), \theta^{k}(t)) \to (\bar{x}(t), 0)$ for almost every $t \in [0, T]$, $\psi^{k} \rightharpoonup - \nabla_{x}\phi(\bar{x}(\cdot), \cdot)$ in $L^{1}([0, T]; \mathbb{R}^{n})$, and $\psi^{k} \in \mathcal{M}(x^{k}, \theta^{k})$ for all $k \in \mathbb{N}$. From the definition of $\mathcal{M}(x^{k}, \theta^{k})$, there exists a sequence $\{(u^{k}, v^{k})\} \subset L^{\infty}([0, T]; \mathbb{R}^{p}\times \mathbb{R}^{m})$ such that, for every $k \in \mathbb{N}$ and almost every $t \in [0, T]$,
\begin{align*}
    & \psi^{k}(t) =\sum_{i=1}^{p}u_{i}^{k}(t)\nabla_{x}h_{i}(x^{k}(t), t) + \sum_{j=1}^{m}v_{j}^{k}(t)\nabla_{x}g_{j}(x^{k}(t), t), \\
    & v_{j}^{k}(t)g_{j}^{-}(x^{k}(t), t) = \theta_{j}^{k}(t), ~ j \in J, \\
    & v_{j}^{k}(t) \geq 0, ~ j \in J.
\end{align*}
Defining $\varepsilon^{k}(t) := \psi^{k}(t) + \nabla_{x}\phi(x^{k}(t), t)$ for all $k \in \mathbb{N}$ and almost every $t \in [0, T]$, it follows that $\bar{x}$ is a pw-AKKT solution of (CTP).
\end{proof}

Inspired by Lemmas \ref{Lem_M_KKT} and \ref{Lema_pw_AKKT}, we define the AKKT-regularity for Problem (CTP) as follows.
\begin{definition} \label{Def_pw_AKKT_regul}
The \emph{asymptotic KKT regularity} (or \emph{AKKT-regularity}) is said to be satisfied at  $\bar{x} \in \Omega$ if
\begin{align*}
\wl\limsup_{(x, \theta) \to (\bar{x}, 0)}\mathcal{M}(x, \theta) \subset \mathcal{M}(\bar{x}, 0).
\end{align*}
\end{definition}

The theorem \ref{Theo_1}, presented below, states that the AKKT-regularity condition serves as a constraint qualification for (CTP), thereby ensuring the validity of the KKT conditions.
\begin{theorem} \label{Theo_1}
Let $\bar{x} \in \Omega$. Assume that (H1) and (H2) are valid and that the AKKT-regularity is satisfied. If $\bar{x}$ is a local optimal solution of (CTP), then it is a KKT solution.
\end{theorem}
\begin{proof}
It follows directly from Theorem \ref{AKKT_opt_cond} and Lemmas \ref{Lem_M_KKT} and \ref{Lema_pw_AKKT}.
\end{proof}

The following example was proposed by \citeauthor{doMonte2026} in \cite[Example 1]{doMonte2026}. It demonstrates a scenario in which the AKKT-regularity condition is violated, yet no KKT solution is present.
\begin{example}
\label{Example_2}
We consider the continuous-time problem below:
\begin{align*}
    \begin{array}{ll}
    \text{minimize} & P(x) = \displaystyle{\int_0^1 \left( t - \frac{1}{2} \right) x_1(t) \dif t} \\
    \text{subject to} & - \left( t - \frac{1}{2} \right) [x_1(t)]^3 + x_2(t) \leq 0 ~ \text{a.e.} ~ t \in [0,1], \\
    & - x_2(t) \leq 0 ~ \text{a.e.} ~ t \in [0,1], \\
    & x \in L^\infty([0,1];\mathbb{R}^2).
    \end{array}
\end{align*}
Observe that $\bar{x}(t) := (0,0)$ a.e. $t \in [0,1]$ is an optimal solution. As demonstrated in \cite[Example 1]{doMonte2026}, the KKT conditions never hold. Since hypotheses (H1) and (H2) are clearly satisfied, it follows from Theorem \ref{Theo_1} that the AKKT-regularity condition does not hold.

The pw-AKKT conditions are satisfied by taking the following sequences
\begin{align*}
    & x_1^k(t) = \left( t - \frac{1}{2} \right) \frac{1}{k}, ~ x_2^k(t) = 0, ~ v_1^k(t) = v_2^k(t) = \frac{k^2}{3 \left( t - \frac{1}{2} \right)^2} ~ \text{a.e.} ~ t \in [0,1] ~ \forall k \in \mathbb{N}.
\end{align*}
\end{example}

The following observation is of particular significance for the subsequent results.

\begin{remark} \label{remark_limit}
Note that, given $\psi \in \mathrm{w}\text{-}\!\limsup_{(x, \theta) \to (\bar{x}, 0)}\mathcal{M}(x, \theta)$, it follows, from the definition of the weak limit, that there exist sequences $\{(x^{k}, \theta^{k})\} \subset E \times F$ and $\{\psi^{k}\} \subset L^{1}([0, T]; \mathbb{R}^{n})$ such that $x^{k}(t) \in B_{r}(t)$ for all $k \in \mathbb{N}$ and $(x^{k}(t), \theta^{k}(t)) \to (\bar{x}(t), 0)$ for almost every $t \in [0, T]$, $\psi^{k} \rightharpoonup \psi$ in $L^{1}([0, T]; \mathbb{R}^{n})$, and $\psi^{k} \in \mathcal{M}(x^{k}, \theta^{k})$ for all $k \in \mathbb{N}$. By the definition of $\mathcal{M}(x^{k}, \theta^{k})$, there exists a sequence $\{(u^{k}, v^{k})\} \subset L^{\infty}([0, T]; \mathbb{R}^{p}\times \mathbb{R}^{m})$ such that,  for every $k \in \mathbb{N}$ and almost every $t \in [0, T]$,
    \begin{align}
        &\psi^{k}(t) =\sum_{i=1}^{p}u_{i}^{k}(t)\nabla_{x}h_{i}(x^{k}(t), t) + \sum_{j=1}^{m}v_{j}^{k}(t)\nabla_{x}g_{j}(x^{k}(t), t),
    \label{Eq_14}\\
        &v_{j}^{k}(t)g_{j}^{-}(x^{k}(t), t) = \theta_{j}^{k}(t), ~ j \in J,
    \label{Eq_15}\\
        &v_{j}^{k}(t) \geq 0, ~ j \in J.
    \label{Eq_16}
    \end{align}
\end{remark}

The next theorem shows that the AKKT-regularity condition is the weakest constraint qualification ensuring that the pw-AKKT conditions implies the KKT conditions for (CTP).
\begin{theorem} \label{Theo_2}
Let $\bar{x}$ be a feasible solution for (CTP). Suppose that (H1) and (H2) are satisfied. Then, the following statements are valid:
\begin{itemize}
\item[(a)] 
Assume that AKKT-regularity is satisfied. If $\bar{x}$ is a pw-AKKT solution, then it is a KKT solution.
\item[(b)] If for every $\phi : \mathbb{R}^n \times [0,T] \rightarrow \mathbb{R}$ satisfying (H1), it holds that ``$\bar{x}$ is a pw-AKKT solution $\Rightarrow$ $\bar{x}$ is a KKT solution'', then AKKT-regularity is satisfied at $\bar{x}$.
\end{itemize}
\end{theorem}
\begin{proof}
Part (a) follows directly from Lemmas \ref{Lem_M_KKT} and \ref{Lema_pw_AKKT}.

Let us prove (b). Let $\psi \in \mathrm{w}\text{-}\!\limsup_{(x, \theta) \to (\bar{x}, 0)}\mathcal{M}(x, \theta)$. 
Consider the function $\phi : \mathbb{R}^n \times [0,T] \rightarrow \mathbb{R}$ given as $\phi(x,t) := - \psi(t) \cdot x$. Define $\varepsilon^{k}(t) := \nabla_{x}\phi(x^{k}(t), t) + \psi^{k}(t)$ for all $k \in \mathbb{N}$ and almost every $t \in [0, T]$, where the sequences $\{x^{k}\}$ and $\{\psi^{k}\}$ are those referred to in Remark \ref{remark_limit}. We have that
\begin{align*} 
\varepsilon^{k}(t) = \nabla_{x}\phi(x^{k}(t), t) + \psi^{k}(t) = -\psi(t) + \psi^{k}(t) ~ \text{a.e.} ~ t \in [0,T] ~ \forall \, k \in \mathbb{N}. 
\end{align*}
Consequently, from Remark \ref{remark_limit}, $\varepsilon^{k} \rightharpoonup 0$. Furthermore, note that
\begin{align*}
& \varepsilon^{k}(t) - \nabla_{x}L(x^{k}(t), u^{k}(t), v^{k}(t), t) \\ 
& \qquad = \psi^{k}(t) - \sum_{i=1}^{p}u_{i}^{k}(t)\nabla_{x}h_{i}(x^{k}(t), t) - \sum_{j=1}^{m}v_{j}^{k}(t)\nabla_{x}g_{j}(x^{k}(t), t) \\
& \qquad = 0  ~ \text{a.e.} ~ t \in [0,T] ~ \forall \, k \in \mathbb{N},
\end{align*}
where we used \eqref{Eq_14} in the last equality. Therefore, $\bar{x}$ is a pw-AKKT solution of (CTP), so that, by hypothesis, $\bar{x}$ is a KKT solution. By Lemma \ref{Lem_M_KKT}, $-\nabla_{x}\phi(\bar{x}(\cdot), \cdot) \in \mathcal{M}(\bar{x}, 0)$, that is, $\psi \in \mathcal{M}(\bar{x}, 0)$. We have just proved that
$$
\wl\limsup_{(x, \theta) \to (\bar{x}, 0)}\mathcal{M}(x, \theta) \subset \mathcal{M}(\bar{x}, 0),
$$
meaning that AKKT-regularity is satisfied at $\bar{x}$.
\end{proof}

\begin{remark} \label{ObsLagAum}
It is imperative to highlight that in \citeauthor{doMonte2026} \cite{doMonte2026}, the authors put forth a variant of the Augmented Lagrange Method that, under specific assumptions, possesses the capacity to generate pw-AKKT sequences. For a more thorough exposition of these results, please refer to \cite[Theorems 3 and 4]{doMonte2026}. Consequently, if AKKT-regularity is satisfied, then the Augmented Lagrange Method is capable of generating sequences that satisfy the KKT conditions. Furthermore, AKKT-regularity is the weakest property under which this result is valid.
\end{remark}

\section{Sufficient criteria for AKKT-regularity} \label{SuffCrit}

In this section, we propose some sufficient criteria for AKKT-regularity. These criteria involve imposing bounds on the Lagrange multipliers, a full rank assumption on the Jacobian matrix of the constraint maps, and the metric regularity of a certain set-valued map defined in terms of the constraints. 

The initial sufficiency criterion is derived by imposing bounds on a given sequence of Lagrange multipliers.

\begin{theorem} \label{Teo_Sufi}
Let $\bar{x} \in \Omega$ . Suppose that (H1) and (H2) are satisfied. Let
\begin{align*}
\psi \in \wl\limsup_{(x, \theta) \to (\bar{x}, 0)}\mathcal{M}(x, \theta).
\end{align*}
Assume that there exist $k_{u} > 0$ and $k_{v} > 0$ such that
\begin{align*}
\Vert u^{k}(t) \Vert \leq k_{u} \quad \text{and} \quad \Vert v^{k}(t) \Vert \leq k_{v} ~ \text{a.e.} ~ t \in [0,T] ~ \forall \, k \in \mathbb{N},
\end{align*}
where the sequence $\{(u^{k},v^{k})\} \in L^\infty([0,T];\mathbb{R}^{p} \times \mathbb{R}^m)$ comes from the definition of the weak-limit above (see Remark \ref{remark_limit}). Then, AKKT-regularity is satisfied at $\bar{x}$.
\end{theorem}
\begin{proof}
    Let $\psi \in \mathrm{w}\text{-} \! \limsup_{(x, \theta) \to (\bar{x}, 0)}\mathcal{M}(x, \theta)$. 
    By hypothesis, there exist constants $k_{u}$ and $k_{v}$ such that $\Vert u^{k}(t)\Vert \leq k_{u}$ and $\Vert v^{k}(t)\Vert \leq k_{v}$ for all $k \in \mathbb{N}$ and almost every $t \in [0, T]$. Hence, $\{(u^{k}, v^{k})\}$ is uniformly bounded in $L^{\infty}([0, T]; \mathbb{R}^{p}\times \mathbb{R}^{m})$. Thus, by extracting subsequences (not renamed), we can conclude that $(u^{k}, v^{k}) \rightharpoonup^{\ast} (u, v)$ in $L^{\infty}([0, T]; \mathbb{R}^{p}\times \mathbb{R}^{m})$. 
    
    From \eqref{Eq_15} and \eqref{Eq_16}, for any Lebesgue measurable set $S \subset [0, T]$, for $j \in J$, we have that
    \begin{align}
    \nonumber
        \int_{S}\theta_{j}^{k}(t)\dif t &= \int_{S}v_{j}^{k}(t)g_{j}^{-}(x^{k}(t), t)\dif t \\
	%
	%
	%
        &= \int_{S}v_{j}^{k}(t)\left[g_{j}^{-}(x^{k}(t), t) - g_{j}^{-}(\bar{x}(t), t)\right]\dif t  + \int_{S}v_{j}^{k}(t)g_{j}^{-}(\bar{x}(t), t) \dif t
    \label{Eq_18} 
    \end{align}
    and
    \begin{align}
        \int_{S}v_{j}^{k}(t)\dif t \geq 0.
    \label{Eq_19}
    \end{align}
It follows from the assumptions that $\{v_{j}^{k}(t)g_{j}^{-}(x^{k}(t), t)\}$ and $\{v_{j}^{k}(t)[g_{j}^{-}(x^{k}(t), t) - g_{j}^{-}(\bar{x}(t), t)]\}$ are dominated sequences. Consequently, by \eqref{Eq_15}, $\{\theta_{j}^{k}(t)\}$ is also a dominated sequence. Thus, by taking the limit as $k \to \infty$ on both sides of \eqref{Eq_18} and \eqref{Eq_19}, using the Dominated Convergence Theorem \cite[Theorem 3.25]{Gordon1994} in the integral on the first term of the equality and in the first integral in the second term, the continuity of $g(\cdot, t)$ for almost every $t \in [0, T]$, and the weak limit $v^{k} \rightharpoonup^{\ast} v$, we obtain
    \begin{align}
        v_{j}(t)g_{j}^{-}(\bar{x}(t), t) = 0, ~ v_{j}(t)\geq 0, ~ j \in J, \text{ a.e } t \in [0, T]. 
    \label{Eq_20}
    \end{align}
    Similarly, using \eqref{Eq_14}, we can write, for any Lebesgue measurable set $S \subset [0, T]$,
    \begin{align*}
        \int_{S}\psi^{k}(t)\dif t &= \int_{S}\sum_{i=1}^{p}u_{i}^{k}(t)\left[\nabla_{x}h_{i}(x^{k}(t), t) - \nabla_{x}h_{i}(\bar{x}(t), t)\right]\dif t\\
        &\quad + \int_{S}\sum_{j=1}^{m}v_{j}^{k}(t)\left[\nabla_{x}g_{j}(x^{k}(t), t) - \nabla_{x}g_{j}(\bar{x}(t), t)\right]\dif t\\
        &\quad + \int_{S}\sum_{i=1}^{p}u_{i}^{k}(t) \nabla_{x}h_{i}(\bar{x}(t), t)\dif t  + \int_{S}\sum_{j=1}^{m}v_{j}^{k}(t) \nabla_{x}g_{j}(\bar{x}(t), t)\dif t
    \end{align*}
    for each $k \in \mathbb{N}$, and, by taking limits, obtain
    \begin{align}
        &\psi(t) = \sum_{i=1}^{p}u_{i}(t)\nabla_{x}h_{i}(\bar{x}(t), t) + \sum_{j=1}^{m}v_{j}(t)\nabla_{x}g_{j}(\bar{x}(t), t) ~ \text{a.e.} ~ t \in [0,T]. 
    \label{Eq_21}
    \end{align}

From \eqref{Eq_20} and \eqref{Eq_21} we see that $\psi \in \mathcal{M}(\bar{x}, 0)$. Therefore, AKKT-regularity is satisfied at $\bar{x}$.
\end{proof}

\begin{remark}
Examples \ref{Example_1} and \ref{Example_2} illustrate situations in which the pw-AKKT condition holds, while the considered solution is not related to the optimal solution of the problem and is not a KKT solution, respectively. In Example \ref{Example_1}, the failure of the AKKT-regularity condition can be ensured by Item (a) of Theorem \ref{Theo_2}. In both cases, the sequences of multipliers considered are not uniformly bounded, that is, the boundedness assumption on the multipliers in Theorem \ref{Teo_Sufi} is not satisfied.
\end{remark}

Following, we present a sufficient criterion for AKKT-regularity under a full rank assumption on the Jacobian of the constraints.

\begin{theorem}
Let $\bar{x} \in \Omega$. Assume that (H1) and (H2) are satisfied. Let 
$$
\psi \in \mathrm{w}\text{-} \!\! \limsup_{(x, \theta) \to (\bar{x}, 0)}\mathcal{M}(x, \theta).
$$
Assume that there exists $k_\psi > 0$ such that 
$$
\Vert \psi^{k}(t) \Vert \leq k_\psi ~ \text{a.e.} ~ t \in [0,T] ~ \forall \, k \in \mathbb{N},
$$
where the sequence $\{\psi^{k}\} \subset L^1([0,T];\mathbb{R}^{n})$ comes from the definition of the weak-limit above (see Remark \ref{remark_limit}). In addition, assume that there exists $K > 0$ such that $\det(\bar{\Upsilon}(t)\bar{\Upsilon}(t)^{\top}) \geq K$ for almost every $t \in [0,T]$, where
\begin{align*}
\bar{\Upsilon}(t) := 
\left[
\begin{array}{c}
\nabla_x h(\bar{x}(t),t) \\ \nabla_x g(\bar{x}(t),t)
\end{array}
\right] ~ \text{a.e.} ~ t \in [0,T].
\end{align*}
Then, AKKT-regularity is satisfied at $\bar{x}$.
\end{theorem}
\begin{proof}
Let $\psi \in \mathrm{w}\text{-} \! \limsup_{(x, \theta) \to (\bar{x}, 0)}\mathcal{M}(x, \theta)$. Let us denote
\begin{align*}
\Upsilon^{k}(t) := \begin{bmatrix} \nabla_{x}h(x^{k}(t),t) \\ \nabla_{x}g(x^{k}(t),t) \end{bmatrix} ~ \text{a.e. } t \in [0,T],
\end{align*}
where the sequence $\{x^{k}\}$ is referred to in Remark \ref{remark_limit}. By the continuity assumptions and the hypothesis, we can assert that
\begin{align} \label{FullRank_k}
\det\left({\Upsilon}^{k}(t){\Upsilon}^{k}(t)^{\top}\right) \geq K ~ \text{a.e. } t \in [0,T].
\end{align}
From \eqref{Eq_14}, we have
\begin{align*}
\psi^k(t) = \Upsilon^{k}(t)^{\top} \begin{bmatrix} u^k(t) \\ v^k(t) \end{bmatrix}
\end{align*}
so that
\begin{align*}
\Upsilon^k(t) \psi^k(t) = \Upsilon^k(t) \Upsilon^k(t)^{\top} \begin{bmatrix} u^k(t) \\ v^k(t) \end{bmatrix}
\end{align*}
and
\begin{align*}
[\Upsilon^k(t)\Upsilon^k(t)^{\top}]^{-1} \Upsilon^k(t) \psi^k(t) = \begin{bmatrix} u^k(t) \\ v^k(t) \end{bmatrix} ~ \text{a.e. } t \in [0,T] ~ \forall \, k \in \mathbb{N}.
\end{align*}
From (H2) and \eqref{FullRank_k}, it follows from \cite[Proposition 3.3]{monte:2019} that
\begin{align*}
\Vert [\Upsilon^k(t)\Upsilon^k(t)^{\top}]^{-1} \Vert \leq \tilde{K} ~ \text{a.e. } t \in [0,T] ~ \forall \, k \in \mathbb{N}.
\end{align*}
Hence,
\begin{align*}
\Vert (u^k(t),v^k(t)) \Vert & \leq \Vert [\Upsilon^k(t)\Upsilon^k(t)^{\top}]^{-1} \Vert \Vert \Upsilon^k(t) \Vert \Vert \psi^k(t) \Vert \\
& \leq \tilde{K} c_{h,g} k_\psi ~ \text{a.e. } t \in [0,T] ~ \forall \, k \in \mathbb{N}.
\end{align*}
Therefore, the result follows from Theorem \ref{Teo_Sufi}.
\end{proof}

The concept of metric regularity, which was introduced to the literature in the field of variational analysis during the 1980s (see \citeauthor{Rocka1998} \cite{Rocka1998}, for instance), also serves as sufficient criteria for AKKT-regularity. Among other significant and notable properties and applications, metric regularity is employed in the context of optimization problems involving non-smooth and/or set-valued maps.

Let $\lambda: \mathbb{R}^{n} \times [0, T] \to \mathbb{R}^{p} \times \mathbb{R}^{m}$ be given as $\lambda(x, t) := (h(x, t), g(x, t))$. Moreover, for each $t \in [0, T]$, let us consider $\Lambda(t) := \{0\} \times \mathbb{R}_{-}^{m_{g}}$. For each $t \in [0, T]$, let us define the set-valued map $\Gamma: L^{1}([0,T];\mathbb{R}^{n}) \rightsquigarrow L^{1}([0,T];\mathbb{R}^{p}) \times L^{1}([0,T];\mathbb{R}^{m})$ by $\Gamma(x)(t) := \Lambda(t) - \{\lambda(x(t), t)\}$. It is easy to see that $x \in \Omega$ is equivalent to $(0,0) \in \Gamma(x)$.

In light of the existing literature, we present an adapted definition of metric regularity which is more appropriate for the current context. In what follows, we will denote the unit ball centred at the origin by $B$, regardless of the space. In any space, the distance function will be denoted simply by $d(\cdot,\cdot)$, and the norm, by $\Vert\cdot\Vert$. It will be clear from the context.

\begin{definition}
    The set-valued mapping $\Gamma$ given above is said to be \emph{metrically regular} at $(\bar{x},(\bar{w},\bar{z})) \in \mathrm{Gr}(\Gamma)$ if there exist $\sigma > 0$ and $\alpha > 0$, such that
    \begin{align*}
        d(x, \Gamma^{-1}(w,z)) \leq \alpha d((w,z),\Gamma(x)) ~ \text{a.e.~in} ~ [0, T]
    \end{align*}
    for all $(x,(w,z)) \in (\bar{x},(\bar{w},\bar{z})) + \sigma B$ with $d((w,z), \Gamma(x)) \leq \sigma$.
\end{definition}

The result below, which establishes a characterization of metric regularity in terms of the coderivative of the set-valued map, will be subsequently employed in establishing one more sufficient criterion for AKKT-regularity. Let $\hat{D}^{\ast}\Gamma$ denotes the Fr\'echet coderivative of $\Gamma$ (see \citeauthor{Mordukhovich2006} \cite[Definition 1.32]{Mordukhovich2006}, for example). 
\begin{lemma} \label{Prop_MR}
    Let $\Gamma$ be the set-valued mapping given above. If $\Gamma$ is metrically regular at $(\bar{x}, (\bar{w},\bar{z})) \in \mathrm{Gr}(\Gamma)$, then there exist $\rho > 0$ and $\epsilon > 0$, such that for all $(x,(w,z)) \in \mathrm{Gr}(\Gamma) \cap ((\bar{x},(\bar{w},\bar{z})) + \epsilon B)$, $(u,v) \in L^{1}([0,T];\mathbb{R}^{p})^{\ast} \times L^{1}([0,T];\mathbb{R}^{m})^{\ast}$ and $\psi \in \hat{D}^{\ast}\Gamma(x,(w,z))(u,v)$, one has
    \begin{align*}
        \rho \Vert (u,v) \Vert \leq \Vert \psi \Vert ~ \text{a.e.~in} ~ [0, T].
    \end{align*}
\end{lemma}
\begin{proof}
    The proof is analogous to the proof of Proposition 4.3 in \citeauthor{Jourani1999} \cite{Jourani1999}, and it will therefore be omitted.
\end{proof}

\begin{remark} \label{RemarkCoderivative}
    By using the properties of the Fréchet coderivative of $\Gamma$ (see, for example, \citeauthor{Mordukhovich2006} \cite[Theorem 1.62]{Mordukhovich2006}), it is easy to see that for all $(x,(w,z)) \in \mathrm{Gr}(\Gamma)$, $(u,v) \in  L^{1}([0,T];\mathbb{R}^{p})^{\ast} \times L^{1}([0,T];\mathbb{R}^{m})^{\ast}$ and $\psi \in \hat{D}^{\ast}\Gamma(x,(w,z))(u,v)$, one has
    $$
    \psi(t) = \sum_{i=1}^{p}u_{i}(t) \nabla_{x} h_{i}(x(t), t) + \sum_{j=1}^{m}v_{j}(t) \nabla_{x} g_{j}(x(t), t) ~ \text{a.e.} ~ t \in [0,T].
    $$
\end{remark}

\begin{theorem}
Let $\bar{x} \in \Omega$. Assume that (H1) and (H2) are satisfied. Let 
$$
\psi \in \wl \limsup_{(x, \theta) \to (\bar{x}, 0)}\mathcal{M}(x, \theta).
$$
Assume that there and $k_\psi > 0$ such that 
$$
\Vert \psi^{k}(t) \Vert \leq k_\psi ~ \text{a.e.} ~ t \in [0,T] ~ \forall \, k \in \mathbb{N},
$$
where the sequences $\{x^{k}\} \subset E$ and $\{\psi^{k}\} \subset L^1([0,T];\mathbb{R}^{n})$ come from the definition of the weak-limit above (see Remark \ref{remark_limit}). In addition, assume that the set-valued map $\Gamma$ defined above is metrically regular at $(\bar{x},-h(\bar{x}(\cdot),\cdot),-g(\bar{x}(\cdot),\cdot)) \in \mathrm{Gr}(\Gamma)$. Then, AKKT-regularity is satisfied at $\bar{x}$.
\end{theorem}
\begin{proof}
Let $\epsilon > 0$ be given by Lemma \ref{Prop_MR} and $\psi \in \mathrm{w}\text{-} \! \limsup_{(x, \theta) \to (\bar{x}, 0)}\mathcal{M}(x, \theta)$. 

We define, for almost every $t \in [0,T]$ and all $k \in \mathbb{N}$,
\begin{align*}
& \bar{w}(t) := - h(\bar{x}(t), t), ~
\bar{z}(t) := -g(\bar{x}(t), t), \\
& w^{k}(t) := - h(x^{k}(t), t), ~
z^{k}(t) := -g(x^{k}(t), t),
\end{align*}
where the sequence $\{x^{k}\}$ is the one referred to in Remark \ref{remark_limit}. It follows that $(w^{k}, z^{k}) \in \Gamma(x^{k})$ for all $k \in \mathbb{N}$. Moreover, since $x^{k}(t) \to \bar{x}(t)$ for almost every $t \in [0, T]$, using (H2), we have that $(w^{k}, z^{k}) \in (\bar{w}, \bar{z}) + \epsilon B$ for all sufficiently large $k \in \mathbb{N}$. From $x^{k}(t) \in B_r(t)$ for all $k \in \mathbb{N}$ and almost every $t \in [0, T]$, we have that $x^{k} \to \bar{x}$ in $L^{1}([0, T]; \mathbb{R}^{n})$, so that $x^{k} \in \bar{x} + \epsilon B$ for all sufficiently large $k \in \mathbb{N}$. Therefore, $(x^{k}, (w^{k}, z^{k})) \in \mathrm{Gr}(\Gamma) \cap \left( (\bar{x}, (\bar{w}, \bar{z})) + \epsilon B \right)$ for all sufficiently large $k \in \mathbb{N}$.
    
From \eqref{Eq_14} and Remark \ref{RemarkCoderivative}, since $\Gamma$ is metrically regular at $(\bar{x},(\bar{w},\bar{z})) \in \mathrm{Gr}(\Gamma)$, by Lemma \ref{Prop_MR}, there exists $\rho > 0$ such that
\begin{align*}
\rho \Vert (u^{k}(t), v^{k}(t))\Vert \leq  \Vert\psi^{k}(t) \Vert ~ \text{a.e} ~ t \in [0, T] ~ \forall \, k \in \mathbb{N}.
\end{align*}
It follows that, 
\begin{align*}
\Vert (u^{k}(t), v^{k}(t)) \Vert \leq \dfrac{k_{\psi}}{\rho} ~ \text{a.e} ~ t \in [0, T] ~ \forall \, k \in \mathbb{N},
\end{align*}
and the result follows from Theorem \ref{Teo_Sufi}.
\end{proof}

\section{Conclusion}

In summary, the present work advances the theory of sequential optimality conditions in continuous-time programming by introducing the AKKT-regularity condition. This constraint qualification, characterized as the minimal assumption under which asymptotic Karush-Kuhn-Tucker (AKKT) conditions imply the classical KKT conditions, provides a rigorous theoretical foundation that strengthens the bridge between sequential and classical optimality frameworks. In this sense, AKKT-regularity plays a role analogous to the Cone Continuity Property (CCP) introduced in \cite{andreani:2016} for nonlinear programming and to its extensions in infinite-dimensional contexts \cite{boorgens:2020}, while being tailored to the specific structure of continuous-time problems.

The AKKT-regularity condition ensures that candidate solutions generated under AKKT conditions are not only meaningful approximations but also converge to genuine stationary points. This guarantee significantly expands the interpretive power of sequential conditions and enhances their applicability in practice. In particular, the proposed constraint qualification provides theoretical support for numerical methods, such as augmented Lagrangian schemes \cite{doMonte2026}, which naturally generate limit points satisfying AKKT conditions but require additional regularity to ensure full optimality.

For future research, it would be natural to explore the interaction between AKKT-regularity and other classes of constraint qualifications in continuous-time optimization. Promising directions include extensions inspired by optimal control, variational analysis, and infinite-dimensional optimization.

We expect that the proposed framework may also serve as a basis for the development of new sequential optimality conditions and constraint qualifications in broader classes of infinite-dimensional optimization problems.

\section*{Acknowledgments}

This research was supported by grants [APQ-00453-21, Minas Gerais Research Foundation (FAPEMIG)], [2022/16005-0, São Paulo Research Foundation (FAPESP)] and [305245/2024-4, National Council for Scientific and Technological Development (CNPq)].

\bibliographystyle{plainnat}
\bibliography{bibliography} 

\end{document}